\documentclass{article}
\usepackage{amsfonts,amsmath,amssymb,amsthm,cite,enumerate,graphicx,latexsym,mathrsfs}
\usepackage{authblk}
\usepackage{hyperref}
\usepackage{float}
\usepackage[pagewise]{lineno}\linenumbers
\makeatletter

\newcommand{\Rmnum}[1]{\expandafter\@slowromancap\romannumeral #1@}
\makeatother
\usepackage{lipsum}
\newcommand\blfootnote[1]{%
\begingroup
\renewcommand\thefootnote{}\footnote{#1}%
\addtocounter{footnote}{-1}%
\endgroup
}
\numberwithin{equation}{section}
\newtheorem{corollary}{Corollary}[section]

\newtheorem{lemma}{Lemma}[section]
\newtheorem{proposition}{Proposition}[section]
\newtheorem{remark}{Remark}[section]

\newtheorem{theorem}{Theorem}[section]

\newcommand\theref[1]{Theorem~\ref{#1}}
\newcommand\lemref[1]{Lemma~\ref{#1}}

\newcommand\proref[1]{Proposition~\ref{#1}}

\newcommand\secref[1]{Section~\ref{#1}}

\title{Subsonic time-periodic solution to the compressible Euler equations triggered by boundary conditions}
\author{Xiaomin Zhang\quad Jiawei Sun \quad Huimin Yu\thanks{Corresponding author. }
 \\ \small\textit{ Department of mathematics, Shandong Normal University, Jinan 250014 China}}

\begin{document}
\begin{sloppypar}
\date{}
\maketitle
\begin{center}
\begin{minipage}{130mm}{\small
\textbf{Abstract}:
In this paper, we consider the one-dimensional isentropic compressible Euler equations with source term $\beta(t,x)\rho|u|^{\alpha}u$ in a bounded domain, which can be used to describe gas transmission in a nozzle.~The model is imposed a subsonic time-periodic boundary condition.~Our main results reveal that the time-periodic boundary can trigger an unique subsonic time-periodic smooth solution and this unique periodic solution is stable under small perturbations on initial and boundary data.~To get the existence of subsonic time-periodic solution, we use the linear iterative skill and transfer the boundary value problem into two initial value ones by using the hyperbolic property of the system. Then the corresponding linearized system can be decoupled.~The uniqueness is a direct by-product of the stability. There is no small assumptions on coefficient $\beta(t,x)$. }
\end{minipage}
\end{center}
\section{Introduction}
\indent\indent We utilize pipes to transfer gas and control its flow in industrial application. Usually, the pipe wall is not smooth sufficiently and there is some resistance, which is regarded as a certain kind of frictional force.~In this paper, the isentropic compressible Euler equations with a friction term is investigated:
\begin{equation}\label{a1}
\left\{\begin{aligned}
&\partial_{t}\rho+\partial_{x}(\rho u)=0,\\
&\partial_{t}(\rho u)+\partial_{x}(\rho u^{2}+p)=\beta(t,x)\rho|u|^{\alpha}u,
\end{aligned}\right.
\quad(t,x)\in\mathbb{R}_{+}\times[0,L],
\end{equation}
where $\rho,u,p$ denote density of mass, velocity of gas and the pressure respectively and $L$ is a positive constant to represent the length of the nozzle. Here we consider the isentropic polytropic gas, i.e.
$$
p=A\rho^{\gamma},
$$
where the adiabatic gas exponent $\gamma>1$ and without loss of generality we assume $A=1$. Moreover, we use $c$ to denote the sonic speed
$$
c=\sqrt{\frac{\partial p}{\partial\rho}}=\sqrt{\gamma}\rho^{\frac{\gamma-1}{2}}.
$$
Suppose the friction coefficient $\beta(t,x)$ is a $C^{1}$ smooth function satisfying
\begin{align}
\beta(t+P,x)=\beta(t,x),\label{F}\\
\|\beta(t,x)\|_{C^{1}(D)}\leq C_{0},\label{F1}
\end{align}
for some constants $P>0$ and $C_{0}>0$. Here the domain $D=\{(t,x)|t\in \mathbb{R}_{+},x\in [0,L]\}$. Obviously, $\beta(t,x)\equiv const.$ satisfies the above conditions~\eqref{F}-\eqref{F1}. Throughout this paper we assume the constant $\alpha>0$. \\
\indent We investigate the global existence and stability of a kind of subsonic time-periodic solution to equations~\eqref{a1}. In recent years,  much effort has been made on the time-periodic solutions to the viscous fluids equations and the hyperbolic conservation laws, see for example~\cite{Cai,Jin,Luo,Ma,Matsumura,G,Ohnawa,Takeno,Temple,Naoki}. However, the time-periodic solutions mentioned above are usually caused by the time-periodic external forces. As far as we know, there is little work to consider the problem with time-periodical boundary. In 2019, Yuan~\cite{Yuan} studied the existence and high-frequency limiting behavior of supersonic time-periodic solutions to the 1-D isentropic compressible Euler equations (i.e. $\beta(t,x)\equiv0$) with time-periodic boundary conditions. \cite{YuH, Yuh} considered the existence and stability of supersonic time periodic flows for the compressible Euler equations with friction term.
It is well known that subsonic boundary condition is more complicated for compressible Euler equation with fixed boundary. Motivated by~\cite{Qu}, in which Qu considered the time-periodic solutions triggered by a kind of dissipative time-periodic boundary condition for the general quasilinear hyperbolic systems, we consider the subsonic time-periodical solutions of isentropic Euler equation with nonlinear source term. It should be noted that: the first author Zhang~\cite{Zhang} studied a similar problem for Euler equation with linear damping and~$\beta(t,x)=\beta(t)$. However, \cite{Zhang} needs the smallness of the friction coefficient $\beta(t)$ and its integral is zero in one period, i.e. $$\displaystyle\int_t^{t+P} \beta(s)ds=0,$$
 which can not be satisfied even for a constant. We remove these two restrictions and include the constant friction coefficient case in this paper.\\
 \indent The rest of this paper is organized as follows. In~\secref{s2}, we first introduce the Riemann invariants of the homogeneous compressible Euler equations and give the main results:~\theref{t1} and~\theref{t2}. In~\secref{s3}, we use the linearized iteration method to prove~\theref{t1}. In~\secref{s4}, using a method similar to that in~\cite{Li}, we first prove the global existence of classical solutions to the initial-boundary problem, then use the inductive method to prove~\theref{t2}.
\section{Prelimilaries and main results}\label{s2}
\indent\indent The two eigenvalues calculated from the system \eqref{a1} are
$$
\lambda_{1}=u-c,\quad\lambda_{2}=u+c.
$$
There holds
\begin{align*}
\lambda_{1}(\underline{\rho},0)<0<\lambda_{2}(\underline{\rho},0)
\end{align*}
and
\begin{align}
\lambda_{1}(\rho,u)<0<\lambda_{2}(\rho,u),\quad (\rho,u)\in\Omega\label{U}
\end{align}
for any positive constant $\underline{\rho}>0$ and a small neighborhood $\Omega$ of $(\underline{\rho},0)$.\\
\indent With the aid of the Riemann invariants $m$ and $n$ defined by
\begin{equation}\label{R}
m=\frac{1}{2}(u-\frac{2}{\gamma-1}c),\quad
n=\frac{1}{2}(u+\frac{2}{\gamma-1}c),
\end{equation}
the equations~\eqref{a1} are changed into the following form
\begin{align}\label{a3}
\left\{
\begin{aligned}
m_{t}+\lambda_{1}(m,n)m_{x}=\frac{\beta(t,x)|m+n|^{\alpha}(m+n)}{2},\\
n_{t}+\lambda_{2}(m,n)n_{x}=\frac{\beta(t,x)|m+n|^{\alpha}(m+n)}{2},
\end{aligned}\right.
\end{align}
where
$$\lambda_{1}(m,n)=\frac{\gamma+1}{2}m+\frac{3-\gamma}{2}n,\quad\lambda_{2}(m,n)=\frac{3-\gamma}{2}m+\frac{\gamma+1}{2}n.$$
\indent Suppose that the solution $(m,n)$ to the system~\eqref{a3} satisfies the following initial data and boundary conditions
\begin{align}
t=0:\quad m(0,x)&=m_{0}(x),~~n(0,x)=n_{0}(x),\label{R1}\\
x=0:\quad n(t,0)&=n_{b}(t),\label{R11}\\
x=L:\quad m(t,L)&=m_{b}(t),\label{R12}
\end{align}
where $m_{b}(t),n_{b}(t)$ are two periodic functions with the period $P>0$.\\
\indent Let
\begin{align*}
\phi(t,x)=(\phi_{1}(t,x),\phi_{2}(t,x))^{\top}&=(m(t,x)-\underline{m},n(t,x)-\underline{n})^{\top},\\
\underline{\phi}&=(\underline{m},\underline{n})^{\top},
\end{align*}
where $\underline{m}=-\frac{1}{\gamma-1}\underline{c}=-\frac{\sqrt{\gamma}}{\gamma-1}\underline{\rho}^{\frac{\gamma-1}{2}},~~\underline{n}=\frac{1}{\gamma-1}\underline{c}
=\frac{\sqrt{\gamma}}{\gamma-1}\underline{\rho}^{\frac{\gamma-1}{2}}$. Then equations~\eqref{a3} can be written as
\begin{align}\label{a4}
\left\{
\begin{aligned}
\partial_{t}\phi_{1}+\lambda_{1}(\phi+\underline{\phi})\partial_{x}\phi_{1}=\frac{\beta(t,x)}{2}|\phi_{1}+\phi_{2}|^{\alpha}(\phi_{1}+\phi_{2}),\\
\partial_{t}\phi_{2}+\lambda_{2}(\phi+\underline{\phi})\partial_{x}\phi_{2}=\frac{\beta(t,x)}{2}|\phi_{1}+\phi_{2}|^{\alpha}(\phi_{1}+\phi_{2})
\end{aligned}\right.
\end{align}
with the corresponding initial data and boundary conditions
\begin{align}
t=0:~~\phi(0,x)&=\phi_{0}(x)=(\phi_{1_{0}}(x),\phi_{2_{0}}(x))^{\top}\notag\\
&\hspace{1.3cm}=(m_{0}(x)-\underline{m},n_{0}(x)-\underline{n})^{\top},\label{a5}\\
x=0:~~\phi_{2}(t,0)&=\phi_{2_{b}}(t)=n_{b}(t)-\underline{n},\quad t\geq0,\label{a6}\\
x=L:~~\phi_{1}(t,L)&=\phi_{1_{b}}(t)=m_{b}(t)-\underline{m},\quad t\geq0.\label{a7}
\end{align}
It is easy to see that $\phi_{i_{b}}(t)(i=1,2)$ are also periodic functions with period $P>0$, i.e. $\phi_{i_{b}}(t+P)=\phi_{i_{b}}(t)$.
We further suppose the following compatibility conditions:
\begin{align}\label{E}
\left\{
\begin{aligned}
&\phi_{1_{b}}(0)=\phi_{1_{0}}(L),~~\phi_{2_{b}}(0)=\phi_{2_{0}}(0),\\
&\phi'_{1_{b}}(0)+\lambda_{1}(\phi_{0}(L)+\underline{\phi})\phi'_{1_{0}}(L)\\
&\hspace{2.5cm}=\frac{\beta(0,L)}{2}\Big(|\phi_{1_{b}}(0)+\phi_{2_{0}}(L)|^{\alpha}
(\phi_{1_{b}}(0)+\phi_{2_{0}}(L))\Big),\\
&\phi'_{2_{b}}(0)+\lambda_{2}(\phi_{0}(0)+\underline{\phi})\phi'_{2_{0}}(0)\\
&\hspace{2.5cm}=\frac{\beta(0,0)}{2}\Big(|\phi_{1_{0}}(0)+\phi_{2_{b}}(0)|^{\alpha}
(\phi_{1_{0}}(0)+\phi_{2_{b}}(0))\Big).
\end{aligned}\right.
\end{align}
\indent By~\eqref{U}, we get
\begin{align}
\lambda_{1}(\phi+\underline{\phi})<0<\lambda_{2}(\phi+\underline{\phi}),\quad \forall \phi\in \Phi,\label{a8}
\end{align}
where $\Phi$ is a small neighborhood of $O=(0,0)^{\top}$ corresponding to $\Omega$.

Define
$$
\nu_{i}(\phi+\underline{\phi})=\lambda_{i}^{-1}(\phi+\underline{\phi}),\quad i=1,2
$$
and denote
$$
\nu_{max}=\mathop{\max}\limits_{i=1,2}\mathop{\sup}\limits_{\phi\in \Phi}|\nu_{i}(\phi+\underline{\phi})|.
$$
By scaling if necessary, we can assume
\begin{align}
\nu_{max}\leq1.\label{D}
\end{align}
Unless specified, in this paper $C_{i}(i=1,2,3,\ldots)$ denotes a generic constant.\\
\indent Next, we give the main results in the following theorems.
\begin{theorem}\label{t1}
(Existence of time-periodic solutions) There exists a small enough constant $\varepsilon_{1}>0$ and a $C^{1}$ smooth function $\phi_{0}=\phi_{0}(x)$, if the $C^{1}$ smooth functions $\phi_{0}(x)$ and $\phi_{i_{b}}(t)(i=1,2)$ satisfy
\begin{align}
\|\phi_{0}\|_{C^{1}([0,L])}\leq C_{1}\varepsilon,\label{y1}\\
\phi_{i_{b}}(t+P)=\phi_{i_{b}}(t),\label{a9}\\
\|\phi_{i_{b}}(t)\|_{C^{1}(\mathbb{R_{+}})}\leq\varepsilon,\label{a10}
\end{align}
for any $\varepsilon\in(0,\varepsilon_{1})$, then the initial-boundary value problem~\eqref{a4}-\eqref{a7} admits a $C^{1}$ time-periodic solution $\phi=\phi^{(P)}(t,x)$ on $D=\{(t,x)|t\in \mathbb{R}_{+},x\in [0,L]\}$ which satisfies
\begin{align}
\phi^{(P)}(t+P,x)&=\phi^{(P)}(t,x),\quad \forall(t,x)\in D,\label{a12}\\
\|\phi^{(P)}\|_{C^{1}(D)}&\leq C_{1}\varepsilon.\label{a13}
\end{align}
\end{theorem}
\begin{theorem}\label{t2}
(Stability of time-periodic solutions) There exists a small constant $\varepsilon_{2}\in(0,\varepsilon_{1})$, such that for any given $\varepsilon\in(0,\varepsilon_{2})$ and any given $C^{1}$ smooth functions $\phi_{0}=\phi_{0}(x)$ and $\phi_{i_{b}}(t)(i=1,2)$ satisfying \eqref{y1} and \eqref{a9}-\eqref{a10} with compatibility conditions \eqref{E}, the initial-boundary value problem \eqref{a4}-\eqref{a7} have a unique global $C^{1}$ classical solution $\phi=\phi(t,x)$ on $D=\{(t,x)|t\in\mathbb{R}_{+},x\in[0,L]\}$ satisfying
\begin{align}
\|\phi(t,\cdot)-\phi^{(P)}(t,\cdot)\|_{C^{0}}\leq C_{2}\varepsilon\xi^{[t/T_{0}]},\quad\forall t\geq0,\label{a15}
\end{align}
where $\phi^{(P)}$, depending on $\phi_{i_{b}}(t)(i=1,2)$, is the time-periodic solution given through \theref{t1}, $\xi\in(0,1)$ is a constant and  $T_{0}=L\nu_{max}$.
\end{theorem}
\indent The uniqueness of the time-periodic solution is a direct consequence from~\theref{t2}.
\begin{corollary}\label{t3}
(Uniqueness of the time-periodic solution) There exists a constant $\varepsilon_{3}\in(0,\varepsilon_{2})$, such that for any given $\varepsilon\in(0,\varepsilon_{3})$ and any given $C^{1}$ smooth functions $\phi_{i_{b}}(t)(i=1,2)$ satisfying~\eqref{a9}-\eqref{a10}, the corresponding time-periodic solution $\phi=\phi^{(P)}(t,x)$ obtained in Theorem 2.1 is unique.
\end{corollary}
\begin{remark}\label{R2}
Theorem \ref{t1} and \ref{t2} are still valid when we change the boundary conditions (\ref{R1})-(\ref{R12}) to 
\begin{align}
t=0:\quad m(0,x)&=m_{0}(x),~~n(0,x)=n_{0}(x),\\
x=0:\quad n(t,0)&=n_{b}(t)+K_1(m(t,0)-\underline{m}),\\
x=L:\quad m(t,L)&=m_{b}(t)+K_2(n(t,L)-\underline{n}),
\end{align}
for $|K_1|<1$ and $|K_2|<1$, where the dissipative structures for the boundary conditions can be kept. The proofs in this manuscript can be extend to this case without any difficulty,
we omit the details here.
\end{remark}
\section{Existence of Time-periodic Solutions}\label{s3}
\indent\indent In this section, we give the proof of~\theref{t1} by applying the linearized iteration method.\\
\indent From~\eqref{a4} and~\eqref{a6}-\eqref{a7}, we consider the following linearized system
\begin{align}
&\partial_{t}\phi_{i}^{(l)}+\lambda_{i}(\phi^{(l-1)}+\underline{\phi})\partial_{x}\phi_{i}^{(l)}=\frac{\beta(t,x)}{2}|\phi_{1}^{(l-1)}+\phi_{2}^{(l-1)}|^{\alpha}
(\phi_{1}^{(l-1)}+\phi_{2}^{(l-1)}),\label{b1}\\
&x=0:~~~~\phi_{2}^{(l)}(t,0)=\left\{
\begin{aligned}
\phi_{2_{b}}(t),\quad t\geq0,\\
\phi_{2_{b'}}(t),\quad t<0,
\end{aligned}\right.\label{b2}\\
&x=L:~~~~\phi_{1}^{(l)}(t,L)=\left\{
\begin{aligned}
\phi_{1_{b}}(t),\quad t\geq0,\\
\phi_{1_{b'}}(t),\quad t<0,
\end{aligned}\right.\label{b3}
\end{align}
where $\phi_{i_{b'}}(t)(i=1,2)$ are obtained by periodic extension of $\phi_{i_{b}}(t)(i=1,2)$. Then we have
\begin{align*}
&\phi_{1}^{(l)}(t,L)=\phi_{1}^{(l)}(t+P,L),\\
&\phi_{2}^{(l)}(t,0)=\phi_{2}^{(l)}(t+P,0)
\end{align*}
for any fixed $t\in\mathbb{R}$. The linearized system~\eqref{b1}-\eqref{b3} is iterated from
\begin{align}
\phi^{(0)}(t,x)=(0,0).\label{b4}
\end{align}
\indent By means of the similar method in~\cite{Yu}, we can show~\theref{t1} from~\proref{p1} below.
\begin{proposition}\label{p1}
There is a small enough constant $\varepsilon_{1}>0$ and a large enough constant $C_{1}>0$, for any given $\varepsilon\in(0,\varepsilon_{1})$, if $\phi_{i_{b}}(t)(i=1,2)$ satisfy~\eqref{a9}-\eqref{a10}, then the sequence of $C^{1}$ solutions $\phi_{i}^{(l)}(t,x)(i=1,2)$ to system~\eqref{b1}-\eqref{b3} satisfy
\begin{align}
&\phi^{(l)}(t+P,x)=\phi^{(l)}(t,x),\quad\forall(t,x)\in D,\quad \forall l\in \mathbb{N_{+}},\label{b5}\\
&\|\phi^{(l)}\|_{C^{1}(D)}\leq C_{1}\varepsilon,\quad\forall l\in \mathbb{N_{+}},\label{b6}\\
&\|\phi^{(l)}-\phi^{(l-1)}\|_{C^{0}(D)}\leq C_{1}\varepsilon\kappa^{l},\quad \forall l\in \mathbb{N_{+}},\label{b7}\\
&\mathop{\max}\limits_{i=1,2}\{\varpi(\delta|\partial_{t}\phi_{i}^{(l)})+\varpi(\delta|\partial_{x}\phi_{i}^{(l)})\}\leq H_{P}(\delta),\quad \forall l\in \mathbb{N_{+}},\label{b8}
\end{align}
where $\kappa\in(0,1)$ ia a constant,
\begin{align}
\|\phi^{(l)}\|_{C^{1}(D)}&\mathop{=}\limits^{\bigtriangleup}\mathop{\max}\limits_{i=1,2}\{\|\phi_{i}^{(l)}\|_{C^{0}(D)},\|\partial_{t}\phi_{i}^{(l)}\|_{C^{0}(D)}
 ,\|\partial_{x}\phi_{i}^{(l)}\|_{C^{0}(D)}\},\notag\\
\varpi(\delta|h)&=\mathop{\sup}\limits_{\mathop{|t_{1}-t_{2}|\leq\delta}\limits_{|x_{1}-x_{2}|\leq\delta}}|h(t_{1},x_{1})-h(t_{2},x_{2})|,\notag
\end{align}
and $H_{P}(\delta)$ is a continuous function of $\delta\in(0,1)$ which is independent of $l$ and satisfies
$$
\mathop{\lim}\limits_{\delta\rightarrow0^{+}}H_{P}(\delta)=0.
$$
\end{proposition}
\begin{proof}
We establish the estimates~\eqref{b5}-\eqref{b8} inductively, i.e., for each $l\in\mathbb{N_{+}}$, we show
\begin{align}
\phi_{i}^{(l)}(t+P,x)=\phi_{i}^{(l)}(t,x),&\quad\forall(t,x)\in D,\forall i=1,2,\label{b9}\\
\mathop{\max}\limits_{{i=1,2}}\|\phi_{i}^{(l)}\|_{C^{1}(D)}&\leq C_{1}\varepsilon,\label{b10}\\
\mathop{\max}\limits_{{i=1,2}}\|\phi_{i}^{(l)}-\phi_{i}^{(l-1)}\|_{C^{0}(D)}&\leq C_{1}\varepsilon\kappa^{l},\label{b11}\\
\mathop{\max}\limits_{{i=1,2}}\varpi(\delta|\partial_{t}\phi_{i}^{(l)}(\cdot,x))&\leq\frac{1}{8}H_{P}(\delta),\quad\forall x\in[0,L]\label{b12}
\end{align}
and
\begin{align}
\mathop{\max}\limits_{{i=1,2}}\{\varpi(\delta|\partial_{t}\phi_{i}^{(l)})+\varpi(\delta|\partial_{x}\phi_{i}^{(l)})\}\leq H_{P}(\delta) \label{b13}
\end{align}
under the following hypothesis
\begin{align}
\phi_{i}^{(l-1)}(t+P,x)=\phi_{i}^{(l-1)}(t,x),&\quad\forall(t,x)\in D,\forall i=1,2,\label{b14}\\
\mathop{\max}\limits_{{i=1,2}}\|\phi_{i}^{(l-1)}\|_{C^{1}(D)}&\leq C_{1}\varepsilon,\label{b15}\\
\mathop{\max}\limits_{{i=1,2}}\|\phi_{i}^{(l-1)}-\phi_{i}^{(l-2)}\|_{C^{0}(D)}&\leq C_{1}\varepsilon\kappa^{l-1},\quad \forall l\geq2,\label{b16}\\
\mathop{\max}\limits_{{i=1,2}}\varpi(\delta|\partial_{t}\phi_{i}^{(l-1)}(\cdot,x))&\leq\frac{1}{8}H_{P}(\delta),\quad\forall x\in[0,L]\label{b17}
\end{align}
and
\begin{align}
\mathop{\max}\limits_{{i=1,2}}\{\varpi(\delta|\partial_{t}\phi_{i}^{(l-1)})+\varpi(\delta|\partial_{x}\phi_{i}^{(l-1)})\}\leq H_{P}(\delta). \label{b18}
\end{align}
Here the constant $\kappa$ should be determined later and
$$\varpi(\delta|h(\cdot,x))=\mathop{\max}\limits_{|t_{1}-t_{2}|\leq\delta}|h(t_{1},x)-h(t_{2},x)|.$$
\indent By~\eqref{b15} with small $\varepsilon>0$, we get that $\phi^{(l-1)}\in \Phi$, which gives the hypothesis~\eqref{a8} is true for~\eqref{b1}-\eqref{b3}. Multiplying $\nu_{i}(\phi^{(l-1)}+\underline{\phi})=\lambda_{i}^{-1}(\phi^{(l-1)}+\underline{\phi})$ on both sides of the $i$-th equation of~\eqref{b1} for $i=1,2$ and swapping the positions of $t$ and $x$, we have
\begin{align}
&\partial_{x}\phi_{1}^{(l)}+\nu_{1}(\phi^{(l-1)}+\underline{\phi})\partial_{t}\phi_{1}^{(l)}\nonumber\\
&=\frac{\beta(t,x)}{2}\nu_{1}(\phi^{(l-1)}+\underline{\phi})|\phi_{1}^{(l-1)}+\phi_{2}^{(l-1)}|^{\alpha}(\phi_{1}^{(l-1)}+\phi_{2}^{(l-1)}),\label{b19}\\
&x=L:~~~~\phi_{1}^{(l)}=\left\{
\begin{aligned}
\phi_{1_{b}}(t),\quad t\geq0,\\
\phi_{1_{b'}}(t),\quad t<0,
\end{aligned}\right.\label{b20}\\
\nonumber\\
&\partial_{x}\phi_{2}^{(l)}+\nu_{2}(\phi^{(l-1)}+\underline{\phi})\partial_{t}\phi_{2}^{(l)}\nonumber\\
&=\frac{\beta(t,x)}{2}\nu_{2}(\phi^{(l-1)}+\underline{\phi})|\phi_{1}^{(l-1)}+\phi_{2}^{(l-1)}|^{\alpha}(\phi_{1}^{(l-1)}+\phi_{2}^{(l-1)}),\label{b21}\\
&x=0:~~~~\phi_{2}^{(l)}=\left\{
\begin{aligned}
\phi_{2_{b}}(t),\quad t\geq0,\\
\phi_{2_{b'}}(t),\quad t<0.
\end{aligned}\right.\label{b22}
\end{align}
\indent Defining the characteristic curves $t=t_{i}^{(l)}(x;t_{0},x_{0})$ for $i=1,2$ and $l\in\mathbb{N_{+}}$ as follows
\begin{align}
\left\{
\begin{aligned}
&\frac{dt_{i}^{(l)}}{dx}(x;t_{0},x_{0})=\nu_{i}(\phi^{(l-1)}+\underline{\phi})(t_{i}^{(l)}(x;t_{0},x_{0}),x),\\
&t_{i}^{(l)}(x_{0};t_{0},x_{0})=t_{0}.
\end{aligned}\right.\label{b23}
\end{align}
\indent First, $\phi_{i}^{(0)}=0(i=1,2)$ satisfy~\eqref{b14}-\eqref{b15} and~\eqref{b17}-\eqref{b18}. Next, we prove estimates~\eqref{b9}-\eqref{b13} for $l\geq1$.\\
\indent By~\eqref{F} and \eqref{b14}, it is easy to check that if $\phi_{i}^{(l)}(t,x)(i=1,2)$ solves problem~\eqref{b19}-\eqref{b22}, so does $\phi_{i}^{(l)}(t+P,x)(i=1,2)$. Then \eqref{b9} is proved by the uniqueness of this linear system. We start to show the $C^{0}$ estimates for $\phi_{i}^{(l)}(i=1,2)$. To do this, we integrate~\eqref{b19} along the characteristic curve $t=t_{1}^{(l)}(x;t_{0},L)$ to obtain
\begin{align*}
&\phi_{1}^{(l)}(t_{1}^{(l)}(x;t_{0},L),x)=\phi_{1}^{(l)}(t_{0},L)\\
&\quad+\int_{L}^{x}\frac{\beta}{2}\nu_{1}(\phi^{(l-1)}+\underline{\phi})|\phi_{1}^{(l-1)}+\phi_{2}^{(l-1)}|^{\alpha}(\phi_{1}^{(l-1)}+\phi_{2}^{(l-1)})(t_{1}^{(l)}(y;t_{0},L),y)dy.
\end{align*}
By~\eqref{F1}, \eqref{D}, \eqref{a10} and~\eqref{b15}, we have
\begin{align}
\|\phi_{1}^{(l)}\|_{C^{0}(D)}\leq\varepsilon+L C_{0}C_{\alpha_{1}}(C_{1}\varepsilon)^{\alpha+1}\leq C_{1}\varepsilon,\label{b25}
\end{align}
where $C_{\alpha_{1}}>0$ is a constant only depending on $\alpha$.\\
\indent In a similar way, integrating~\eqref{b21} along $t=t_{2}^{(l)}(x;t_{0},0)$ to obtain
\begin{align}
\|\phi_{2}^{(l)}\|_{C^{0}(D)}\leq C_{1}\varepsilon.\label{b26}
\end{align}
From~\eqref{b25} and~\eqref{b26}, we obtain the $C^{0}$ norm estimates of $\phi_{i}^{(l)}$ in~\eqref{b10} is true.\\
\indent Denote $\varphi_{i}^{(l)}(i=1,2)$ as
\begin{align}
\varphi_{i}^{(l)}=\partial_{t}\phi_{i}^{(l)},\quad i=1,2,~~l\in\mathbb{N_{+}}.\label{b27}
\end{align}
Differentiating equations~\eqref{b19} and~\eqref{b21} with respect to $t$, we obtain
\begin{align}
&\partial_{x}\varphi_{1}^{(l)}+\nu_{1}(\phi^{(l-1)}+\underline{\phi})\partial_{t}\varphi_{1}^{(l)}\nonumber\\
&=-\Big(\nabla\nu_{1}(\phi^{(l-1)}+\underline{\phi})\cdot\partial_{t}\phi^{(l-1)}\Big)\varphi_{1}^{(l)}\nonumber\\
&\quad+\frac{\alpha+1}{2}\beta\nu_{1}(\phi^{(l-1)}+\underline{\phi})|\phi_{1}^{(l-1)}+\phi_{2}^{(l-1)}|^{\alpha}(\partial_{t}\phi_{1}^{(l-1)}+\partial_{t}\phi_{2}^{(l-1)})\nonumber\\
&\quad+\frac{\beta}{2}\Big(\nabla\nu_{1}(\phi^{(l-1)}+\underline{\phi})\cdot\partial_{t}\phi^{(l-1)}\Big)|\phi_{1}^{(l-1)}+\phi_{2}^{(l-1)}|^{\alpha}(\phi_{1}^{(l-1)}+\phi_{2}^{(l-1)})\nonumber\\
&\quad+\frac{\partial_{t}\beta}{2}\nu_{1}(\phi^{(l-1)}+\underline{\phi})|\phi_{1}^{(l-1)}+\phi_{2}^{(l-1)}|^{\alpha}(\phi_{1}^{(l-1)}+\phi_{2}^{(l-1)}),\label{b28}
\end{align}
\begin{align}
&\partial_{x}\varphi_{2}^{(l)}+\nu_{2}(\phi^{(l-1)}+\underline{\phi})\partial_{t}\varphi_{2}^{(l)}\nonumber\\
&=-\Big(\nabla\nu_{2}(\phi^{(l-1)}+\underline{\phi})\cdot\partial_{t}\phi^{(l-1)}\Big)\varphi_{2}^{(l)}\nonumber\\
&\quad+\frac{\alpha+1}{2}\beta\nu_{2}(\phi^{(l-1)}+\underline{\phi})|\phi_{1}^{(l-1)}+\phi_{2}^{(l-1)}|^{\alpha}(\partial_{t}\phi_{1}^{(l-1)}+\partial_{t}\phi_{2}^{(l-1)})\nonumber\\
&\quad+\frac{\beta}{2}\Big(\nabla\nu_{2}(\phi^{(l-1)}+\underline{\phi})\cdot\partial_{t}\phi^{(l-1)}\Big)|\phi_{1}^{(l-1)}+\phi_{2}^{(l-1)}|^{\alpha}(\phi_{1}^{(l-1)}+\phi_{2}^{(l-1)})\nonumber\\
&\quad+\frac{\partial_{t}\beta}{2}\nu_{2}(\phi^{(l-1)}+\underline{\phi})|\phi_{1}^{(l-1)}+\phi_{2}^{(l-1)}|^{\alpha}(\phi_{1}^{(l-1)}+\phi_{2}^{(l-1)}).\label{b29}
\end{align}
Integrating~\eqref{b28} along the $1$-characteristic curve $t=t_{1}^{(l)}(x;t_{0},L)$, we have
\begin{align*}
&\varphi_{1}^{(l)}(t_{1}^{(l)}(x;t_{0},L),x)\\
&=\varphi_{1}^{(l)}(t_{0},L)-\int_{L}^{x}\Big(\nabla\nu_{1}(\phi^{(l-1)}+\underline{\phi})\cdot\partial_{t}\phi^{(l-1)}\Big)\varphi_{1}^{(l)}(t_{1}^{(l)}(y;t_{0},L),y)dy\\
&\quad+\int_{L}^{x}\Big(\frac{\alpha+1}{2}\beta\nu_{1}(\phi^{(l-1)}+\underline{\phi})|\phi_{1}^{(l-1)}+\phi_{2}^{(l-1)}|^{\alpha}(\partial_{t}\phi_{1}^{(l-1)}+\partial_{t}\phi_{2}^{(l-1)})\\
&\quad+\frac{\beta}{2}\Big(\nabla\nu_{1}(\phi^{(l-1)}+\underline{\phi})\cdot\partial_{t}\phi^{(l-1)}\Big)|\phi_{1}^{(l-1)}+\phi_{2}^{(l-1)}|^{\alpha}(\phi_{1}^{(l-1)}+\phi_{2}^{(l-1)})\\
&\quad+\frac{\partial_{t}\beta}{2}\nu_{1}(\phi^{(l-1)}+\underline{\phi})|\phi_{1}^{(l-1)}+\phi_{2}^{(l-1)}|^{\alpha}(\phi_{1}^{(l-1)}+\phi_{2}^{(l-1)})\Big)(t_{1}^{(l)}(y;t_{0},L),y)dy.
\end{align*}
By $\nu_{i}(\phi+\underline{\phi})=\lambda_{i}^{-1}(\phi+\underline{\phi})$ and~\eqref{D}, we have
\begin{align}
\mathop{\sup}\limits_{\phi\in \Phi}|\nabla\nu_{i}(\phi+\underline{\phi})|\leq\nu_{max}^{2}\mathop{\sup}\limits_{\phi\in \Phi}|\nabla\lambda_{i}(\phi+\underline{\phi})|\leq\mathop{\sup}\limits_{\phi\in \Phi}|\nabla\lambda_{i}(\phi+\underline{\phi})|
\leq C_{3},\label{b30}
\end{align}
where constant $C_{3}>0$ is independent of $l$.\\
\indent By~\eqref{F1}, \eqref{D}, \eqref{a10}, \eqref{b15} and~\eqref{b30}, we obtain
\begin{align}
\|\varphi_{1}^{(l)}\|_{C^{0}(D)}\leq&(\varepsilon+2LC_{0}C_{\alpha_{2}}(C_{1}\varepsilon)^{\alpha+1}+L\alpha C_{0}C_{\alpha_{2}}(C_{1}\varepsilon)^{\alpha+1}\nonumber\\
&+LC_{0}C_{3}C_{\alpha_{1}}(C_{1}\varepsilon)^{\alpha+2}
+LC_{0}C_{\alpha_{1}}(C_{1}\varepsilon)^{\alpha+1})e^{LC_{1}C_{3}\varepsilon}\nonumber\\
\leq&\hbar C_{1}\varepsilon,\label{b31}
\end{align}
where $C_{\alpha_{2}}>0$ is a constant only depending on $\alpha$ and constant $0<\hbar<1$ is independent of $l$.\\
\indent In a similar way, we integrate~\eqref{b29} along the $2$-characteristic curve $t=t_{2}^{(l)}(x;t_{0},0)$ to obtain
\begin{align}
\|\varphi_{2}^{(l)}\|_{C^{0}(D)}\leq \hbar C_{1}\varepsilon.\label{b32}
\end{align}
By applying the equations~\eqref{b19} and~\eqref{b21} and noting~\eqref{F1}, \eqref{D}, \eqref{b15} and~\eqref{b31}-\eqref{b32}, we gain
\begin{align}
\|\partial_{x}\phi_{i}^{(l)}\|_{C^{0}(D)}\leq\hbar C_{1}\varepsilon+C_{0}C_{\alpha_{1}}(C_{1}\varepsilon)^{\alpha+1}\leq C_{1}\varepsilon.\label{b33}
\end{align}
By~\eqref{b4}, we can get~\eqref{b11} for $l=1$ directly from~\eqref{b25}-\eqref{b26}. We begin to prove~\eqref{b11} for $l\geq2$. It follows from~\eqref{b19} that
\begin{align}
&\Big(\partial_{x}+\nu_{1}(\phi^{(l-1)}+\underline{\phi})\partial_{t}\Big)(\phi_{1}^{(l)}-\phi_{1}^{(l-1)})\nonumber\\
&=-\Big(\nu_{1}(\phi^{(l-1)}+\underline{\phi})-\nu_{1}(\phi^{(l-2)}+\underline{\phi})\Big)\partial_{t}\phi_{1}^{(l-1)}+\frac{\beta}{2}\nu_{1}(\phi^{(l-1)}+\underline{\phi})\nonumber\\
&\quad\times\Big[|\phi_{1}^{(l-1)}+\phi_{2}^{(l-1)}|^{\alpha}(\phi_{1}^{(l-1)}+\phi_{2}^{(l-1)})
-|\phi_{1}^{(l-2)}+\phi_{2}^{(l-2)}|^{\alpha}(\phi_{1}^{(l-2)}+\phi_{2}^{(l-2)})\Big]\nonumber\\
&\quad+\frac{\beta}{2}\Big(\nu_{1}(\phi^{(l-1)}+\underline{\phi})-\nu_{1}(\phi^{(l-2)}+\underline{\phi})\Big)|\phi_{1}^{(l-2)}+\phi_{2}^{(l-2)}|^{\alpha}(\phi_{1}^{(l-2)}+\phi_{2}^{(l-2)}).\label{b34}
\end{align}
Integrating~\eqref{b34} along the $1-$characteristic curve $t=t_{1}^{(l)}(x;t_{0},L)$, we have
\begin{align*}
&|\phi_{1}^{(l)}(t_{1}^{(l)}(x;t_{0},L),x)-\phi_{1}^{(l-1)}(t_{1}^{(l)}(x;t_{0},L),x)|\\
\leq&|\phi_{1}^{(l)}(t_{0},L)-\phi_{1}^{(l-1)}(t_{0},L)|\\
&+L\mathop{\sup}\limits_{\phi\in \Phi}|\nabla\nu_{1}(\phi+\underline{\phi})|\|\phi^{(l-1)}-\phi^{(l-2)}\|_{C^{0}(D)}\|\partial_{t}\phi^{(l-1)}\|_{C^{0}(D)}\\
&+L\mathop{\sup}|\beta|\nu_{\max}(\alpha+1)|\phi_{1}^{(l-1)}+\phi_{2}^{(l-1)}|^{\alpha}\|\phi^{(l-1)}-\phi^{(l-2)}\|_{C^{0}(D)}\\
&+\frac{1}{2}L|\mathop{\sup}\beta|\mathop{\sup}\limits_{\phi\in \Phi}|\nabla\nu_{1}(\phi+\underline{\phi})|\|\phi^{(l-1)}-\phi^{(l-2)}\|_{C^{0}(D)}|\phi_{1}^{(l-2)}+\phi_{2}^{(l-2)}|^{\alpha+1}.
\end{align*}
By~\eqref{F1}, \eqref{D}, \eqref{b15}-\eqref{b16}, \eqref{b20} and~\eqref{b30}, it follows that
\begin{align}
\|\phi_{1}^{(l)}-\phi_{1}^{(l-1)}\|_{C^{0}(D)}\leq& LC_{3}(C_{1}\varepsilon)^{2}\kappa^{l-1}+2L(\alpha+1)C_{0}C_{\alpha_{2}}(C_{1}\varepsilon)^{\alpha+1}\kappa^{l-1}\nonumber\\
&+LC_{0}C_{3}C_{\alpha_{1}}(C_{1}\varepsilon)^{\alpha+2}\kappa^{l-1}\nonumber\\
\leq& C_{1}\varepsilon^{1+\eta}\kappa^{l-1},\label{b35}
\end{align}
where $\eta=\min\{1,\alpha\}$. In a similar way, we can also get
\begin{align}
\|\phi_{2}^{(l)}-\phi_{2}^{(l-1)}\|_{C^{0}(D)}\leq C_{1}\varepsilon^{1+\eta}\kappa^{l-1}.\label{b36}
\end{align}
Thus, by choosing $\kappa\in(0,1)$ satisfying $\varepsilon^{\eta}<\kappa$, we obtain~\eqref{b11} from~\eqref{b35}-\eqref{b36}.\\
\indent We will prove~\eqref{b12} for $l\geq1$. Let
$$
H_{P}(\delta)=20\mathop{\max}\limits_{i=1,2}\varpi(\delta|\phi'_{i_{b}})+20\delta
$$
for some constant $\delta\in(0,1)$. Because of $\phi'_{i_{b}}\in C^{0}(\mathbb{R}_{+})$, we have $\mathop{\lim}\limits_{\delta\rightarrow0^{+}}H_{P}(\delta)=0$.\\
\indent We derive from the definition~\eqref{b23} that for any two points $(t_{1},x_{0})$ and $(t_{2},x_{0})$ in the domain $D$ with $|t_{1}-t_{2}|\leq\delta$ and $x_{0}\in[0,L]$,
\begin{align*}
&|t_{1}^{(l)}(x;t_{1},x_{0})-t_{1}^{(l)}(x;t_{2},x_{0})|\\
&\leq|t_{1}-t_{2}|+|\int_{x_{0}}^{x}\nu_{1}(\phi^{(l-1)}+\underline{\phi})(t_{1}^{(l)}(y;t_{1},x_{0}),y)\\
&\hspace{2.2cm}-\nu_{1}(\phi^{(l-1)}+\underline{\phi})(t_{1}^{(l)}(y;t_{2},x_{0}),y)dy|\\
&\leq|t_{1}-t_{2}|+|\int_{x_{0}}^{x}|t_{1}^{(l)}(y;t_{1},x_{0})-t_{1}^{(l)}(y;t_{2},x_{0})|\\
&\hspace{2.2cm}|\int_{0}^{1}\Big(\nabla\nu_{1}(\phi^{(l-1)}+\underline{\phi})\cdot\partial_{t}
\phi^{(l-1)}\Big)\\
&\hspace{2.2cm}\Big(\gamma t_{1}^{(l)}(y;t_{1},x_{0})+(1-\gamma)t_{1}^{(l)}(y;t_{2},x_{0}),y\Big)d\gamma dy|.
\end{align*}
By Gronwall's inequality, \eqref{b15} and~\eqref{b30}, it follows that
\begin{align}
|t_{1}^{(l)}(x;t_{1},x_{0})-t_{1}^{(l)}(x;t_{2},x_{0})|\leq&|t_{1}-t_{2}|e^{LC_{1}C_{3}\varepsilon}\nonumber\\
\leq&\delta(1+LC_{1}C_{3}\varepsilon),\quad\forall x\in[0,L].\label{b37}
\end{align}
Integrating~\eqref{b28} along $t=t_{1}^{(l)}(x;t_{1},x_{0})$ and $t=t_{1}^{(l)}(x;t_{2},x_{0})$ respectively and then subtracting the two resulted expressions, we get
\begin{align*}
&\varphi_{1}^{(l)}(t_{2},x_{0})-\varphi_{1}^{(l)}(t_{1},x_{0})\\
&=\varphi_{1}^{(l)}(t_{1}^{(l)}(L;t_{2},x_{0}),L)-\varphi_{1}^{(l)}(t_{1}^{(l)}(L;t_{1},x_{0}),L)\\
&\quad+\int_{L}^{x_{0}}-\Big(\nabla\nu_{1}(\phi^{(l-1)}+\underline{\phi})\cdot\partial_{t}\phi^{(l-1)}\Big)\varphi_{1}^{(l)}|_{(t_{1}^{(l)}(x;t_{1},x_{0}),x)}
^{(t_{1}^{(l)}(x;t_{2},x_{0}),x)}dx\\
&\quad+\int_{L}^{x_{0}}\Big(\frac{\alpha+1}{2}\beta\nu_{1}(\phi^{(l-1)}+\underline{\phi})|\phi_{1}^{(l-1)}+\phi_{2}^{(l-1)}|^{\alpha}(\partial_{t}\phi_{1}^{(l-1)}+\partial_{t}\phi_{2}^{(l-1)})\\
&\quad+\frac{\beta}{2}\Big(\nabla\nu_{1}(\phi^{(l-1)}+\underline{\phi})\cdot\partial_{t}\phi^{(l-1)}\Big)|\phi_{1}^{(l-1)}+\phi_{2}^{(l-1)}|^{\alpha}(\phi_{1}^{(l-1)}+\phi_{2}^{(l-1)})\\
&\quad+\frac{\partial_{t}\beta}{2}\nu_{1}(\phi^{(l-1)}+\underline{\phi})|\phi_{1}^{(l-1)}+\phi_{2}^{(l-1)}|^{\alpha}(\phi_{1}^{(l-1)}+\phi_{2}^{(l-1)})\Big)
|_{(t_{1}^{(l)}(x;t_{1},x_{0}),x)}^{(t_{1}^{(l)}(x;t_{2},x_{0}),x)}dx.
\end{align*}
By the Gronwall's inequality, \eqref{F1}, \eqref{D}, \eqref{b15}, \eqref{b30} and~\eqref{b37}, it follows that
\begin{align}
\varpi(\delta|\varphi_{1}^{(l)}(\cdot,x))\leq \frac{1}{8}H_{P}(\delta),\quad\forall x\in[0,L].\label{b38}
\end{align}
In a similar way, we have
\begin{align}
\varpi(\delta|\varphi_{2}^{(l)}(\cdot,x))\leq \frac{1}{8}H_{P}(\delta),\quad\forall x\in[0,L].\label{b39}
\end{align}
From~\eqref{b38}-\eqref{b39}, we obtain~\eqref{b12}.\\
\indent Finally, we prove~\eqref{b13}. We first consider the special case that two given points $(t_{1},x_{1})$ and $(t_{2},x_{2})$ with $|t_{1}-t_{2}|\leq\delta,|x_{1}-x_{2}|\leq\delta$ located on the same characteristic curve $t=t_{1}^{(l)}(x;t_{0},x_{0})$, namely, $t_{2}=t_{1}^{(l)}(x_{2};t_{1},x_{1})$. \\
\indent Integrating~\eqref{b28} along the characteristic curve $t=t_{1}^{(l)}(x;t_{1},x_{1})$, we have
\begin{align*}
&\varphi_{1}^{(l)}(t_{2},x_{2})-\varphi_{1}^{(l)}(t_{1},x_{1})\\
&=\int_{x_{1}}^{x_{2}}\Big(-\Big(\nabla\nu_{1}(\phi^{(l-1)}+\underline{\phi})\cdot\partial_{t}\phi^{(l-1)}\Big)\varphi_{1}^{(l)}\\
&\quad+\frac{\alpha+1}{2}\beta\nu_{1}(\phi^{(l-1)}+\underline{\phi})|\phi_{1}^{(l-1)}+\phi_{2}^{(l-1)}|^{\alpha}(\partial_{t}\phi_{1}^{(l-1)}+\partial_{t}\phi_{2}^{(l-1)})\\
&\quad+\frac{\beta}{2}\Big(\nabla\nu_{1}(\phi^{(l-1)}+\underline{\phi})\cdot\partial_{t}\phi^{(l-1)}\Big)|\phi_{1}^{(l-1)}+\phi_{2}^{(l-1)}|^{\alpha}(\phi_{1}^{(l-1)}+\phi_{2}^{(l-1)})\\
&\quad+\frac{\partial_{t}\beta}{2}\nu_{1}(\phi^{(l-1)}+\underline{\phi})|\phi_{1}^{(l-1)}+\phi_{2}^{(l-1)}|^{\alpha}(\phi_{1}^{(l-1)}+\phi_{2}^{(l-1)})\Big)(t_{1}^{(l)}(x;t_{1},x_{1}),x)dx.
\end{align*}
By~\eqref{F1}, \eqref{D}, \eqref{b15} and~\eqref{b30}, we have
\begin{align}
|\varphi_{1}^{(l)}(t_{2},x_{2})-\varphi_{1}^{(l)}(t_{1},x_{1})|\leq\frac{1}{12}H_{P}(\delta).\label{b40}
\end{align}
For two given points $(t_{1},x_{1})$ and $(t_{2},x_{2})$ with $|t_{1}-t_{2}|\leq\delta,|x_{1}-x_{2}|\leq\delta$ located on characteristic curves $t=t_{1}^{(l)}(x;t_{1},x_{1})$ and $t=t_{2}^{(l)}(x;t_{2},x_{2})$ respectively, we select a point $(t_{3},x_{2})$ on the $1$-th characteristic curve passing through $(t_{1},x_{1})$, namely,
$$
t_{3}=t_{1}^{(l)}(x_{2};t_{1},x_{1}).
$$
By~\eqref{D} and definition~\eqref{b23}, one has
$$
|t_{3}-t_{1}|\leq|\nu_{1}||x_{2}-x_{1}|\leq|x_{2}-x_{1}|\leq\delta,
$$
thus
$$
|t_{3}-t_{2}|\leq|t_{3}-t_{1}|+|t_{2}-t_{1}|\leq2\delta.
$$
By means of the estimates~\eqref{b38} and~\eqref{b40}, we have
\begin{equation}\label{b400}
\begin{aligned}
&|\varphi_{1}^{(l)}(t_{2},x_{2})-\varphi_{1}^{(l)}(t_{1},x_{1})|\\
&\leq|\varphi_{1}^{(l)}(t_{2},x_{2})-\varphi_{1}^{(l)}(\frac{t_{2}+t_{3}}{2},x_{2})|+|\varphi_{1}^{(l)}(\frac{t_{2}+t_{3}}{2},x_{2})-\varphi_{1}^{(l)}(t_{3},x_{2})|\\
&\quad+|\varphi_{1}^{(l)}(t_{3},x_{2})-\varphi_{1}^{(l)}(t_{1},x_{1})|\\
&\leq\frac{1}{4}H_{P}(\delta)+\frac{1}{12}H_{P}(\delta)\\
&=\frac{1}{3}H_{P}(\delta).
\end{aligned}
\end{equation}
The combination of~\eqref{b400} and~\eqref{b40} leads to
\begin{align}
\varpi(\delta|\varphi_{1}^{(l)})\leq\frac{1}{3}H_{P}(\delta).\label{b41}
\end{align}
In a similar way, we obtain
\begin{align}
\varpi(\delta|\varphi_{2}^{(l)})\leq\frac{1}{3}H_{P}(\delta).\label{b42}
\end{align}
By the aid of~\eqref{b19} and~\eqref{b21}, \eqref{F1}, \eqref{D}, \eqref{b15}, \eqref{b30} and~\eqref{b41}-\eqref{b42}, we have
\begin{align}
\varpi(\delta|\partial_{x}\phi_{i}^{(l)})\leq\frac{1}{2}H_{P}(\delta),\quad i=1,2. \label{b43}
\end{align}
Hence, \eqref{b13} is proved from~\eqref{b41}\eqref{b43}. The proof of~\proref{p1} is completed.
\end{proof}
\indent Under the help of~\proref{p1} and the similar arguments as in \cite{Qu}, the proof of Theorem 2.1 could be presented, here we omit the details.
\section{Stability of the Time-periodic Solution}\label{s4}
\indent\indent In this section, we give the proof of~\theref{t2} to consider the stability of the time-periodic solution obtained in Theorem 2.1. For the sake of proving the existence of the classical solutions $\phi=\phi(t,x)$, we only need to prove the following ~\lemref{l1} on the basis of the existence and uniqueness of local $C^{1}$ solution for the mixed initial-boundary value problem for quasilinear hyperbolic system(cf. Chapter $4$ in \cite{Yu}).
Inspired by the method in~\cite{Li}, we give the proof of~\lemref{l1}.
\begin{lemma}\label{l1}
 There exists a small constant $\varepsilon_{0}>0$, for any given $\varepsilon\in(0,\varepsilon_{0})$, there exists $\sigma=\sigma(\varepsilon)>0$ such that if
 \begin{align}
 &\|\phi_{i_{b}}\|_{C^{1}(\mathbb{R}_{+})}\leq\sigma,\quad i=1,2,\label{c111}\\
 &\|\phi_{0}\|_{C^{1}[0,L]}\leq\sigma,\label{c22}
 \end{align}
 then the $C^{1}$ solution $\phi=\phi(t,x)$ to the initial-boundary value problem~\eqref{a4}-\eqref{a7} satisfies
\begin{align}
\|\phi\|_{C^{1}(D)}\leq \varepsilon. \label{c2}
\end{align}
\end{lemma}
\begin{proof}
We make a-priori assumption on the $C^{1}$ solution $\phi=\phi(t,x)$ as follows
\begin{align}
\|\phi\|_{C^{1}(D)}\leq \varepsilon_{0}\label{G}
\end{align}
for some constant $\varepsilon_{0}>0$.\\
\indent By~\eqref{a8}, \eqref{G} and choosing $\varepsilon_{0}$ small enough, we can have on domain $D$
\begin{align}
\lambda_{1}(t,x)<0<\lambda_{2}(t,x).\label{c3}
\end{align}
Let
\begin{align}
T_{0}&=\mathop{\max}\limits_{1\leq i\leq2}\mathop{\sup}\limits_{\phi\in \Phi}\frac{L}{|\lambda_{i}(\phi+\underline{\phi})|}=L\nu_{max},\label{a16}\\
\lambda_{max}&=\mathop{\max}\limits_{\mathop{(t,x)\in D}\limits_{i=1,2}}|\lambda_{i}(t,x)|,\quad
T_{1}=\frac{L}{\lambda_{max}}.\notag
\end{align}
\indent By continuity and~\eqref{a5}, \eqref{c22}, there exists a small constant $\sigma>0$ such that~\eqref{c2} holds on a domain $D(\eta_{0})=\{(t,x)|0\leq t\leq\eta_{0},0\leq x\leq L\}$ for a suitably small constant $\eta_{0}>0$. Therefore, for the sake of proving~\eqref{c2}, we only need to show that there exists a small constant $\varepsilon_{0}>0$ such that if~\eqref{c2} holds on the domain $D(T)=\{(t,x)|0\leq t\leq T, 0\leq x\leq L\}$ for any fixed $T>0$, then it still hold on the domain $\{(t,x)|T\leq t\leq T+T_{1}, 0\leq x\leq L\}$ for any $T_{1}>0$.\\
\indent Let
\begin{align*}
w(t,x)&=(w_{1}(t,x),w_{2}(t,x))^{\top},\\
w_{i}(t,x)&=\partial_{x}\phi_{i}(t,x),\quad i=1,2.
\end{align*}
Differentiating~\eqref{a4} with respect to $x$, we get
\begin{align}
&\partial_{t}w_{i}+\lambda_{i}(\phi+\underline{\phi})\partial_{x}w_{i}\nonumber\\
&=-(\nabla\lambda_{i}(\phi+\underline{\phi})\cdot w)w_{i}+\frac{\alpha+1}{2}\beta|\phi_{1}+\phi_{2}|^{\alpha}(w_{1}+w_{2})\nonumber\\
&\quad+\frac{\partial_{x}\beta}{2}
|\phi_{1}+\phi_{2}|^{\alpha}(\phi_{1}+\phi_{2})\label{k1}
\end{align}
with $i=1,2$.\\
\indent Define
\begin{align}
&\phi(t)=\mathop{\max}\limits_{\mathop{i=1,2}\limits_{0\leq x\leq L}}|\phi_{i}(t,x)|,\label{c4}\\
&w(t)=\mathop{\max}\limits_{\mathop{i=1,2}\limits_{0\leq x\leq L}}|w_{i}(t,x)|.\label{c4l}
\end{align}
For the $C^{1}$ solution exists on $D(T+T_{1})=\{(t,x)|0\leq t\leq T+T_{1}, 0\leq x\leq L\}$, we define the $i$-th characteristic curve $g=g_{i}(\tau;t,x)$ passing through a point $(t,x)\in D(T+T_{1})$ with $t\in[T,T+T_{1}]$ by the following form
\begin{align*}
\left\{\begin{aligned}
&\frac{dg_{i}(\tau;t,x)}{d\tau}=\lambda_{i}(\tau,g_{i}(\tau;t,x)),\\
&\tau=t:g_{i}(t;t,x)=x.
\end{aligned}\right.
\end{align*}
Each of the characteristic curves $g_{i}(\tau;t,x), i=1,2$ has two possibilities.  The $1$-th characteristic curve is described as an example.\\
\textbf{Case 1:} The $1$-th characteristic curve $g=g_{1}(\tau;t,x)$ intersects the interval $[0,L]$ on the $x$-axis with an intersection point $(0,g_{1}(0;t,x))$, see Figure $1$.
\begin{figure}[H]
\centering
\includegraphics[width=9cm]{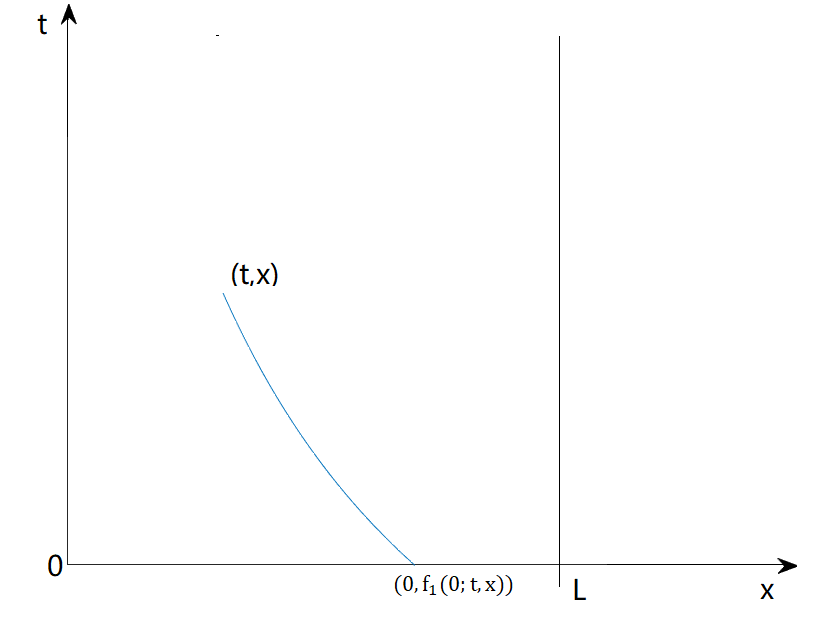}
\begin{center}
Figure 1\quad Image of the $1$-th characteristic curve intersecting  the $x$-axis.
\end{center}
\end{figure}
\noindent\textbf{Case 2:} The $1$-th characteristic curve $g=g_{1}(\tau;t,x)$ intersects the boundary $x=L$ with an intersection point $(\tau_{1}(t,x),L)$, where $\tau_{1}(t,x)$ satisfies
$$
g_{1}(\tau_{1}(t,x);t,x)=L,
$$
see Figure $2$.\\
\indent Clearly, one has
$$
T_{0}\geq t-\tau_{1}(t,x)\geq0.
$$
\begin{figure}[H]
\centering
\includegraphics[width=9cm]{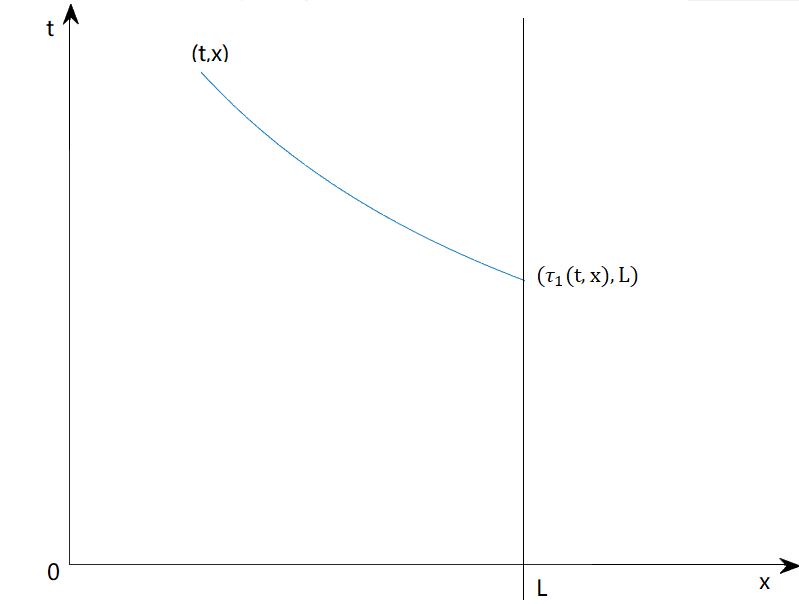}
\begin{flushleft}
Figure 2\quad Image of the $1$-th characteristic curve intersecting the boundary $x=L$.
\end{flushleft}
\end{figure}
We first prove the $C^{0}$ estimates of ~\eqref{c2}. For Case 1, integrating the $1$-th equation in~\eqref{a4} along the $1$-th characteristic $g=g_{1}(\tau;t,x)$, we obtain
$$
\phi_{1}(t,x)=\phi_{1_{0}}(x)+\int_{0}^{t}\frac{\beta}{2}|\phi_{1}+\phi_{2}|^{\alpha}(\phi_{1}+\phi_{2})(\tau,g_{1}(\tau;t,x))d\tau.
$$
By~\eqref{F1}, \eqref{c22}, \eqref{G} and~\eqref{c4}, we get
\begin{align}
|\phi_{1}(t,x)|\leq \sigma+2C_{0}C_{\alpha_{2}}\varepsilon_{0}^{\alpha}\int_{0}^{t}\phi(\tau)d\tau.\label{c5}
\end{align}
Since $t\leq T_{0}$ in this case, one has $T\leq T_{0}$. Noting that~\eqref{c2} holds on $D(T)$, we get
\begin{align}
|\phi_{1}(t,x)|\leq \sigma+2T_{0}C_{0}C_{\alpha_{2}}\varepsilon_{0}^{\alpha}\varepsilon+2C_{0}C_{\alpha_{2}}\varepsilon_{0}^{\alpha}\int_{T}^{t}\phi(\tau)d\tau.\label{c6}
\end{align}
\indent For Case $2$, integrating the $1$-th equation in~\eqref{a4} along the $1$-th characteristic $g=g_{1}(\tau;t,x)$ and using~\eqref{a7}, we have
\begin{align}
\phi_{1}(t,x)=\phi_{1_{b}}(\tau_{1}(t,x))+\int_{\tau_{1}(t,x)}^{t}\frac{\beta}{2}|\phi_{1}+\phi_{2}|^{\alpha}(\phi_{1}+\phi_{2})(\tau,g_{1}(\tau;t,x))d\tau.\label{c7}
\end{align}
By~\eqref{F1}, \eqref{c111}, \eqref{G} and~\eqref{c4}, it follows from~\eqref{c7} that
\begin{align}
|\phi_{1}(t,x)|\leq \sigma+2C_{0}C_{\alpha_{2}}\varepsilon_{0}^{\alpha}\int_{\tau_{1}(t,x)}^{t}\phi(\tau)d\tau.\label{cc7}
\end{align}
No matter wether $\tau_{1}(t,x)>T$ or $\tau_{1}(t,x)\leq T$, similar to~\eqref{c6}, we have
\begin{align}
|\phi_{1}(t,x)|\leq \sigma+2T_{0}C_{0}C_{\alpha_{2}}\varepsilon_{0}^{\alpha}\varepsilon+ 2C_{0}C_{\alpha_{2}}\varepsilon_{0}^{\alpha}\int_{T}^{t}\phi(\tau)d\tau.\label{c8}
\end{align}
\indent By~\eqref{c6}and~\eqref{c8}, we obtain for any fixed point $(t,x)\in D(T+T_{1})$ with $t\in[T,T+T_{1}]$
$$
|\phi_{1}(t,x)|\leq \sigma+2T_{0}C_{0}C_{\alpha_{2}}\varepsilon_{0}^{\alpha}\varepsilon+ 2C_{0}C_{\alpha_{2}}\varepsilon_{0}^{\alpha}\int_{T}^{t}\phi(\tau)d\tau.
$$
By the arguments as above, the similar estimates of $\phi_{2}(t,x)$ can also be derived. \\
\indent Thus, we have
$$
\phi(t)\leq \sigma+2T_{0}C_{0}C_{\alpha_{2}}\varepsilon_{0}^{\alpha}\varepsilon+ 2C_{0}C_{\alpha_{2}}\varepsilon_{0}^{\alpha}\int_{T}^{t}\phi(\tau)d\tau,\quad\forall t\in[T,T+T_{1}].
$$
By Gronwall's inequality, we get
\begin{align}
\phi(t)\leq(\sigma+2T_{0}C_{0}C_{\alpha_{2}}\varepsilon_{0}^{\alpha}\varepsilon)e^{2C_{0}C_{\alpha_{2}}\varepsilon_{0}^{\alpha}T_{1}}\leq \varepsilon,\quad\forall t\in[T,T+T_{1}].\label{c9l}
\end{align}
Therefore, we get the $C^{0}$ estimates of~\eqref{c2} by choosing suitably small positive constants $\varepsilon_{0}$ and $\sigma$.\\
\indent Next, we show the estimates~\eqref{c2} for the spatial derivative of $\phi$. For Case 1, we integrate the $1$-th equation in~\eqref{k1} to get
\begin{align*}
w_{1}(t,x)=&w_{1_{0}}(x)+\int_{0}^{t}\big(-(\nabla\lambda_{1}(\phi+\underline{\phi})\cdot w)w_{1}\\
&+\frac{\alpha+1}{2}\beta|\phi_{1}+\phi_{2}|^{\alpha}(w_{1}+w_{2})\\
&+\frac{\partial_{x}\beta}{2}|\phi_{1}+\phi_{2}|^{\alpha}(\phi_{1}+\phi_{2})\big)(\tau,g_{1}(\tau;t,x))d\tau.
\end{align*}
By~\eqref{F1}, \eqref{b30}, \eqref{c22}, \eqref{G}, \eqref{c4l} and~\eqref{c9l}, we have
\begin{align}
|w_{1}(t,x)|\leq \sigma+C_{0}C_{\alpha_{1}}\varepsilon^{\alpha+1}(T+T_{1})+(C_{3}\varepsilon_{0}+ C_{0}C_{\alpha_{2}}\varepsilon^{\alpha})\int_{0}^{t}w(\tau)d\tau.\label{c10l}
\end{align}
Since $t\leq T_{0}$ in this case, one has $T\leq T_{0}$. Noting that~\eqref{c2} holds on $D(T)$, we get
\begin{align}
|w_{1}(t,x)|\leq& \sigma+C_{0}C_{\alpha_{1}}\varepsilon^{\alpha+1}(T+T_{1})+
T_{0}C_{3}\varepsilon_{0}\varepsilon+T_{0} C_{0}C_{\alpha_{2}}\varepsilon^{\alpha+1}\nonumber\\
&+(C_{3}\varepsilon_{0}+C_{0}C_{\alpha_{2}}\varepsilon^{\alpha})\int_{T}^{t}w(\tau)d\tau.\label{c11l}
\end{align}
\indent For Case $2$, we need to give the boundary condition $\partial_{x}\phi_{1}(t,L)$. By~$\eqref{a4}_{1}$, we have
\begin{align*}
\partial_{x}\phi_{1}(t,L)=-\frac{1}{\lambda_{1}(\phi+\underline{\phi})(t,L)}\phi'_{1b}(t)+\frac{\beta(t,L)}{2\lambda_{1}(\phi+\underline{\phi})(t,L)}
\big(|\phi_{1}+\phi_{2}|^{\alpha}(\phi_{1}+\phi_{2})\big)(t,L)
\end{align*}
which implies from~\eqref{F1}, \eqref{D}, \eqref{c111} and~\eqref{c9l}
\begin{align}
|\partial_{x}\phi_{1}(t,L)|\leq C_{4}(\sigma+\varepsilon^{\alpha+1}),\quad \forall t\in(0,T+T_{1}]. \label{HB1}
\end{align}
Integrating the $1$-th equation in~\eqref{k1} along the $1$-th characteristic $g=g_{1}(\tau;t,x)$, we have
\begin{align}
w_{1}(t,x)=&w_{1}(\tau_{1}(t,x),L)+\int_{\tau_{1}(t,x)}^{t}\big(-(\nabla\lambda_{1}(\phi+\underline{\phi})\cdot w)w_{1}\nonumber\\
&+\frac{\alpha+1}{2}\beta|\phi_{1}+\phi_{2}|^{\alpha}(w_{1}+w_{2})\nonumber\\
&+\frac{\partial_{x}\beta}{2}|\phi_{1}+\phi_{2}|^{\alpha}(\phi_{1}+\phi_{2})\big)(\tau,g_{1}(\tau;t,x))d\tau.\label{c12l}
\end{align}
By~\eqref{F1}, \eqref{b30}, \eqref{G}, \eqref{c4l}, \eqref{c9l} and~\eqref{HB1}, it follows from~\eqref{c12l} that
\begin{align*}
|w_{1}(t,x)|\leq&C_{4}(\sigma+\varepsilon^{\alpha+1})+ C_{0}C_{\alpha_{1}}\varepsilon^{\alpha+1}(T+T_{1})\notag\\
&+(C_{3}\varepsilon_{0}+ C_{0}C_{\alpha_{2}}\varepsilon^{\alpha})\int_{\tau_{1}(t,x)}^{t}w(\tau)d\tau.
\end{align*}
No matter wether $\tau_{1}(t,x)>T$ or $\tau_{1}(t,x)\leq T$, similar to~\eqref{c11l}, we have
\begin{align}
|w_{1}(t,x)|\leq& C_{4}(\sigma+\varepsilon^{\alpha+1})+C_{0}C_{\alpha_{1}}\varepsilon^{\alpha+1}(T+T_{1})+T_{0}C_{3}\varepsilon_{0}\varepsilon+T_{0} C_{0}C_{\alpha_{2}}\varepsilon^{\alpha+1}\nonumber\\
&+(C_{3}\varepsilon_{0}+C_{0}C_{\alpha_{2}}\varepsilon^{\alpha})\int_{T}^{t}w(\tau)d\tau.\label{c13l}
\end{align}
\indent By~\eqref{c11l}and~\eqref{c13l}, we have
\begin{align*}
|w_{1}(t,x)|\leq& C_{4}(\sigma+\varepsilon^{\alpha+1})+C_{0}C_{\alpha_{1}}\varepsilon^{\alpha+1}(T+T_{1})+T_{0}C_{3}\varepsilon_{0}\varepsilon+T_{0} C_{0}C_{\alpha_{2}}\varepsilon^{\alpha+1}\\
&+(C_{3}\varepsilon_{0}+ C_{0}C_{\alpha_{2}}\varepsilon^{\alpha})\int_{T}^{t}w(\tau)d\tau.
\end{align*}
By the arguments as above, the similar estimates for $w_{2}(t,x)$ can also be obtained. \\
\indent Thus, we have
\begin{align*}
w(t)\leq& C_{4}(\sigma+\varepsilon^{\alpha+1})+C_{0}C_{\alpha_{1}}\varepsilon^{\alpha+1}(T+T_{1})+T_{0}C_{3}\varepsilon_{0}\varepsilon+T_{0} C_{0}C_{\alpha_{2}}\varepsilon^{\alpha+1}\\
&+ (C_{3}\varepsilon_{0}+ C_{0}C_{\alpha_{2}}\varepsilon^{\alpha})\int_{T}^{t}w(\tau)d\tau,\quad \forall t\in[T,T+T_{1}].
\end{align*}
Then by Gronwall's inequality, we have
\begin{align}
w(t)\leq& (C_{4}\sigma+C_{4}\varepsilon^{\alpha+1}+C_{0}C_{\alpha_{1}}\varepsilon^{\alpha+1}(T+T_{1})+T_{0}C_{3}\varepsilon_{0}\varepsilon+T_{0} C_{0}C_{\alpha_{2}}\varepsilon^{\alpha+1})\nonumber\\
&\times e^{(C_{3}\varepsilon_{0}+ C_{0}C_{\alpha_{2}}\varepsilon^{\alpha})T_{1}}\nonumber\\
\leq& \varepsilon,\quad \forall t\in[T,T+T_{1}],\label{c14l}
\end{align}
where we have chosen suitably small positive constants $\varepsilon_{0}$ and $\sigma$ for the above last inequality. Moreover, by the above similar steps, we can also get
\begin{align}
\|\partial_{t}\phi\|_{C^{0}}\leq \varepsilon.\label{c15l}
\end{align}
Combining~\eqref{c9l}, \eqref{c14l} and~\eqref{c15l}, we complete the proof of~\eqref{c2}. Meanwhile, this also indicates that the previous hypothesis~\eqref{G} is rational.
\end{proof}
\indent Under the help of the local existence and uniqueness of the classical solution stated in~\cite{Yu} and the continuity argument, we obtain from~\lemref{l1} the global existence and uniqueness of the classical solutions $\phi=\phi(t,x)$ to the initial-boundary value problem~\eqref{a4}-\eqref{a7} with
\begin{align}
\|\phi\|_{C^{1}(D)}\leq C_{1}\varepsilon.\label{c9}
\end{align}
Next, we prove~\eqref{a15} inductively. Suppose that for some $t_{*}>0$ and $N\in\mathbb{N}$, we have
\begin{align}
\mathop{\max}\limits_{i=1,2}\|\phi_{i}(t,\cdot)-\phi_{i}^{(P)}(t,\cdot)\|_{C^{0}}\leq C_{2}\varepsilon\xi^{N},\quad\forall t\in[t_{*},t_{*}+T_{0}], \label{Indu1}
\end{align}
we will prove
\begin{align}
\mathop{\max}\limits_{i=1,2}\|\phi_{i}(t,\cdot)-\phi_{i}^{(P)}(t,\cdot)\|_{C^{0}}\leq C_{2}\varepsilon\xi^{N+1},\quad\forall t\in[t_{*}+T_{0},t_{*}+2T_{0}],\label{Indu2}
\end{align}
where $\xi\in(0,1)$ is a constant to be determined later and $\phi_{i}^{(P)}(t,x), i=1,2$ is the time-periodic solution proved in~\theref{t1}. Let
$$
\theta(t)=\mathop{\max}\limits_{1\leq i\leq2}\mathop{\sup}\limits_{x\in[0,L]}|\phi_{i}(t,x)-\phi_{i}^{(P)}(t,x)|.
$$
From~\eqref{Indu1}, it follows that $\theta(t)$ is continuous and
$$
\theta(t_{*}+T_{0})\leq C_{2}\varepsilon\xi^{N}.
$$
It's just necessary to prove
\begin{align}
\theta(t)\leq C_{2}\varepsilon\xi^{N+1},\quad\forall t\in[t_{*}+T_{0},\tau]\label{c10}
\end{align}
under the hypothesis
\begin{align}
\theta(t)\leq C_{2}\varepsilon\xi^{N},\quad\forall t\in[t_{*},\tau]\label{c11}
\end{align}
for any $\tau\in[t_{*}+T_{0},t_{*}+2T_{0}]$.\\
\indent We have from~\eqref{a4}
\begin{align}
\Big(\partial_{t}+\lambda_{i}(\phi+\underline{\phi})\partial_{x}\Big)\phi_{i}&=\frac{\beta(t,x)}{2}|\phi_{1}+\phi_{2}|^{\alpha}(\phi_{1}+\phi_{2}),\quad i=1,2,\label{c12}\\
\Big(\partial_{t}+\lambda_{i}(\phi^{(P)}+\underline{\phi})\partial_{x}\Big)\phi_{i}^{(P)}&=\frac{\beta(t,x)}{2}|\phi_{1}^{(P)}+\phi_{2}^{(P)}|^{\alpha}(\phi_{1}^{(P)}+\phi_{2}^{(P)}),\quad i=1,2.\label{c13}
\end{align}
Since $\phi$ and $\phi_{i}^{(P)}$ satisfy the same boundary condition, \eqref{c10} holds obviously at the boundary $x=L$ and $x=0$.\\
\indent By~\eqref{c12}-\eqref{c13}, we get
\begin{align*}
&\Big(\partial_{t}+\lambda_{i}(\phi+\underline{\phi})\partial_{x}\Big)(\phi_{i}-\phi_{i}^{(P)})\\
=&\Big(\lambda_{i}(\phi^{(P)}+\underline{\phi})-\lambda_{i}(\phi+\underline{\phi})\Big)\partial_{x}\phi_{i}^{(P)}+\frac{\beta}{2}\Big(|\phi_{1}+\phi_{2}|^{\alpha}(\phi_{1}+\phi_{2})\\
&-|\phi_{1}^{(P)}+\phi_{2}^{(P)}|^{\alpha}(\phi_{1}^{(P)}+\phi_{2}^{(P)})\Big),\quad i=1,2.
\end{align*}
By~\eqref{a16}, the backward characteristic curve $x=g_{i}(t;\hat{t},\hat{x})$ passing through any point $(\hat{t},\hat{x})\in[t_{*}+T_{0},\tau]\times[0,L]$ intersects the boundary $x=0$ or $x=L$ at $t\in[t_{*},\tau]$. Integrating the above equation along the $i$-th characteristic curve $x=g_{i}(t;\hat{t},\hat{x})~(i=1,2)$, by~\eqref{F1}, \eqref{a13}, \eqref{b30} and~\eqref{c11}, we have
\begin{align*}
\theta(\hat{t})\leq&2T_{0}\mathop{\sup}\limits_{\phi\in \Phi}|\nabla\lambda_{i}(\phi+\underline{\phi})|\|\phi^{(P)}-\phi\|_{C^{0}}\|\partial_{x}\phi_{i}^{(P)}\|_{C^{0}}\\
&+2T_{0}|\beta|(\alpha+1)|\phi_{1}+\phi_{2}|^{\alpha}\|\phi_{i}-\phi_{i}^{(P)}\|_{C^{0}}\\
\leq&2T_{0}C_{1}C_{3}\varepsilon C_{2}\varepsilon\xi^{N}+4T_{0}(\alpha+1)C_{0}C_{\alpha_{2}}(C_{1}\varepsilon)^{\alpha}C_{2}\varepsilon\xi^{N}.
\end{align*}
\indent We select small $\varepsilon_{2}>0$ and some constant $\xi\in(0,1)$ such that
$$
2T_{0}C_{1}C_{3}\varepsilon+4T_{0}(\alpha+1)C_{0}C_{\alpha_{2}}(C_{1}\varepsilon)^{\alpha}\leq\xi,
$$
for any $\varepsilon\in(0,\varepsilon_{2})$, which gives us
$$
\theta(\hat{t})\leq C_{2}\varepsilon\xi^{N+1}.
$$
Since $\hat{t}$ is arbitrary, we obtain~\eqref{c10}. Thus we complete the proof of~\theref{t2}.
\section{Acknowledge}\label{s5}
\indent\indent The authors would like to give many thanks to Professor Hairong Yuan for his encouragement and discussion.~This work is supported in part by the National Natural Science Foundation of China (Grant No. 11671237).

\begin{align*}
E-mail~address:&~~Xiaomin~zhang: zxm15924687@163.com\\
&~~Jiawei~Sun: sunjiawei0122@163.com\\
&~~Huimin~Yu: hmyu@sdnu.edu.cn
\end{align*}

\end{sloppypar}
\end{document}